\documentclass[11pt,oneside,a4paper]{article}

\usepackage[latin1]{inputenc}
\usepackage[english]{babel}
\usepackage{amsmath,amssymb,amsthm,epsfig,color,graphicx}
\usepackage{hyperref}
\usepackage{vmargin}

\newtheorem{theorem}{Theorem}
\newtheorem{proposition}[theorem]{Proposition}
\newtheorem{lemma}[theorem]{Lemma}

\newtheorem{remark}{Remark}

\title{Existence and concentration of positive bound states for the Schrodinger-Poisson system with potential functions}

\date{}

\author{Patr\'{i}cia L. Cunha \thanks{Supported by CAPES-PROEX/Brazil.}\thanks{patcunha80@gmail.com} \\
\footnotesize{Departamento de Informática e Métodos Quantitativos, FGV-SP, Brazil}\\ }

\begin{document}
\maketitle

\begin{abstract}
In this article we study the existence and concentration behavior of bound states for a nonlinear Schrödinger-Poisson system with a parameter
$\varepsilon>0$. Under some suitable conditions on the potential functions, we prove that for $\varepsilon$ small the system has a positive
solution that concentrates at a point which is a global minimum of the minimax function associated to the related autonomous problem.
\end{abstract}

\text{\footnotesize{\textit{Keywords}: Schrödinger-Poisson system; variational methods, concentration.}}

\section{Introduction}

In this article we will focus on the following Schrödinger-Poisson system
\begin{equation*} \left \{ \begin{array}{ll}
 -\varepsilon^2\Delta v+ V(x)v+K(x)\phi(x) v= |v|^{q-2}v & \mbox{in}\quad \mathbb{R}^{3}\\
-\Delta \phi=K(x)v^{2} & \mbox{in}\quad \mathbb{R}^{3}
\end{array}\right.\tag{$\mathcal{SP}_{\varepsilon}$}
\end{equation*}
where $\varepsilon>0$ is a parameter, $q\in(4,6)$ and $V,K:\mathbb{R}^3\rightarrow\mathbb{R}$ are, respectively, an external potential and a charge
density. The unknowns of the system are the field $u$ associated with the particles and the electric potential $\phi$. We are interested in the existence
 and concentration behavior of solutions of $(\mathcal{SP}_\varepsilon)$ in the semiclassical limit $\varepsilon\rightarrow 0$.

The first equation of ($\mathcal{SP}_{\varepsilon}$) is a nonlinear equation in which the potential $\phi$ satisfies a nonlinear Poisson equation.
For this reason, ($\mathcal{SP}_{\varepsilon}$) is called a Schrödinger-Poisson system, also known as Schrödinger-Maxwell system. For more informations about
physical aspects, we refer \cite{Benci-Fortunato-1998,D'Aprile-Mugnai} and references therein.

We observe that when $\phi\equiv 0$, ($\mathcal{SP}_\varepsilon$) reduces to the well known Schrödinger equation
\begin{equation*} -\varepsilon^2\Delta u+ V(x)u = f(x,u) \quad x\in\mathbb{R}^N \tag{$\mathcal{S}$}.
\end{equation*}

In the last years, the nonlinear stationary Schrödinger equation has been widely investigated, mainly in the semiclassical limit as
$\varepsilon\rightarrow 0$ (see e.g. \cite{Rabinowitz,Wang,Wang-Zeng} and its references).

In \cite{Rabinowitz}, Rabinowitz studied problem ($\mathcal{S}$) through mountain pass arguments in order to find least energy solutions, for
$\varepsilon>0$  sufficiently small. Then, Wang \cite{Wang} proved that the solution in \cite{Rabinowitz}  concentrates around the global minimal of
$V$ when  $\varepsilon$ tends to 0.

In \cite{Wang-Zeng}, Wang and Zeng considered the following  Schrödinger equation
\begin{equation*}-\varepsilon^2\Delta u+V(x)u =K(x) |u|^{p-1}u+Q(x) |u|^{q-1}u, \quad x\in\mathbb{R}^N\tag{$\mathcal{WZ}$}
\end{equation*}
where $1<q<p<(n+2)/{(n-2)^+}$. They proved the existence of least energy solutions and their concentration around a point in the semiclassical limit.
The authors used the energy function $C(s)$ defined as the minimal energy of the functional associated with
$\Delta u+V(s)u =K(s) |u|^{p-1}u+Q(s) |u|^{q-1}u$,
where $s\in\mathbb{R}^N$ acts as a parameter instead of an independent variable. For each $\varepsilon>0$ sufficiently small, they proved the
existence of a solution $u_\varepsilon$ for ($\mathcal{WZ}$), whose global maximum  approaches to a point $y^*$  when $\varepsilon$ tends to 0.
Moreover, under suitable hypothesis on the potentials $V$ e $W$, the function $\xi\mapsto C(\xi)$ assumes a minimum at $y^*$.

Motivated by those results, Alves and Soares \cite{Alves-Soares} investigated the same phenomenon for the following class of gradient systems
\begin{equation*} \left \{ \begin{array}{ll}
 -\varepsilon^2\Delta u+ V(x)u= Q_u(u,v) & \mbox{in}\quad \mathbb{R}^{N}\\
 -\varepsilon^2\Delta v+ W(x)v= Q_v(u,v) & \mbox{in}\quad \mathbb{R}^{N}\\
u(x), v(x)\rightarrow 0, \,\,\, \mbox{ao}\,\, |x|\rightarrow\infty &\\
u,v>0 \quad \mathbb{R}^{N}
\end{array}\right.\tag{$\mathcal{AS}$}
\end{equation*}

In this system is natural to expect some competition between the potentials $V$ and $W$, each one trying to attract the local maximum
points of the solutions to its minimum points. In fact, in \cite{Alves-Soares} the authors proved that functions $u_\varepsilon$ and $v_\varepsilon$
satisfies ($\mathcal{AS}$) and concentrate around the same point which is the minimum of the respective function $C(s)$.

In \cite{Yang-Han}, Yang and Han studied the following Schrödinger-Poisson system
\begin{equation*} \left \{ \begin{array}{ll}
\Delta v+ V(x)v+K(x)\phi(x) v= |v|^{q-2}v & \mbox{in}\quad \mathbb{R}^{3}\\
-\Delta \phi=K(x)v^{2} & \mbox{in}\quad \mathbb{R}^{3}
\end{array}\right.\tag{$\mathcal{SP}$}.
\end{equation*}
Under suitable assumptions on $V$, $K$ and $f$ they proved existence and multiplicity results by using the mountain pass theorem and the fountain theorem.
Later, L. Zhao, Liu and F. Zhao \cite{Zhao-Liu-Zhao}, using variational methods, proved the existence and concentration of solutions for system
\begin{equation*} \left \{ \begin{array}{ll}
\Delta v+ \lambda V(x)v+K(x)\phi(x) v= |v|^{q-2}v & \mbox{in}\quad \mathbb{R}^{3}\\
-\Delta \phi=K(x)v^{2} & \mbox{in}\quad \mathbb{R}^{3}
\end{array}\right.
\end{equation*}
when $\lambda>0$ is a parameter and $2<p<6$.

Several papers dealt with system $(\mathcal{SP})$ under variety assumptions on potentials $V$ and $K$. Most part of the literature focuses on the study
of the system with $V$ or $K$ constant or radially symmetric, mainly studying existence, nonexistence and multiplicity of solutions see e.g.
\cite{Ambrosetti,Coclite,D'Aprile-Mugnai,Fang-Zhang,Kikuchi,Mercuri,Ruiz}.

The double parameters'perturbation was also considered in system $(\mathcal{SP}_\varepsilon)$. In \cite{He-Zou2012}, He and Zhou studied the existence and
 behavior of a ground state solution which concentrates around the global minimum of the potential V. They considered $K\equiv 1$ and the presence of the
 nonlinear term $f(x,u)$.

Recently, Ianni and Vaira \cite{Ianni-Vaira} studied the Schrödinger-Poisson system $(\mathcal{SP}_\varepsilon)$ proving that if $V$ has a
non-degenerated critical
point $x_0$, then there exists a solution that concentrates around this point. Moreover, they also proved that if $x_0$
is degenerated for $V$ and a local minimum for $K$, then there exist a solution concentrating around $x_0$. The proof was based in the
Lyapunov-Schmidt reduction.

Using variational methods as employed by \cite{Alves-Soares,Rabinowitz,Wang-Zeng}, we prove that there exists a solution $u_\varepsilon$ for the Schrödinger-Poisson
system ($\mathcal{SP}_\varepsilon $) which concentrates around a point, without any additional assumption on the degenerability of such point related with the
potentials $V$ and $K$, as used in \cite{Ianni-Vaira}.

More precisely, denote $C_\infty$ as the minimax value related to
\begin{eqnarray*} \left \{ \begin{array}{ll}
 -\Delta v+ V_\infty v+K_\infty\phi v= |v|^{q-2}v & \mbox{in}\quad \mathbb{R}^{3}\\
-\Delta \phi=K_\infty v^{2} & \mbox{in}\quad \mathbb{R}^{3}
\end{array}\right.
\end{eqnarray*}
where the following conditions hold
\begin{itemize}
\item [$(H_0)$] There exists $\alpha>0$ such that $V(x), K(x)\geq \alpha>0$, \,\, $\forall\,x\in\mathbb{R}^3$
\item [$(H_1)$] $V_\infty$ and $K_\infty$ are defined by \begin{equation*}  \begin{array}{ll}
V_\infty=\liminf _{|x|\rightarrow\infty}V(x)>\inf_{x\in\mathbb{R}^3}V(x)\\
K_\infty=\liminf _{|x|\rightarrow\infty}K(x)>\inf_{x\in\mathbb{R}^3}K(x).
\end{array}\end{equation*}
\end{itemize}

We prove that if
\begin{eqnarray*}C_\infty>\inf_{\xi\in\mathbb{R}^3} C(\xi)
\end{eqnarray*}
then, system ($\mathcal{SP}_{\varepsilon}$) has a positive solution $v_\varepsilon$ as $\varepsilon$ tends to zero. After passing to a subsequence,
$v_\varepsilon$ concentrates at a global minimum point of $C(\xi)$ for $\xi\in\mathbb{R}^3$, where the energy function $C(\xi)$ is defined to be the
minimax function associated with the problem
\begin{equation*} \left \{ \begin{array}{ll}
 -\Delta u+ V(\xi)u+K(\xi)\phi(\xi) u= |u|^{q-2}u & \mbox{in}\quad \mathbb{R}^{3}\\
-\Delta \phi=K(\xi)u^{2} & \mbox{in}\quad \mathbb{R}^{3}
\end{array}\right.\tag{$\mathcal{SP}_{\xi}$}
\end{equation*}
Therefore, $C(\xi)$ plays a central role in our study.

The main result for system ($\mathcal{SP}_{\varepsilon}$) is the following

\begin{theorem}\label{principal} Suppose $(H_0)-(H_1)$ hold. If
\begin{equation}\tag{$C^\infty$} C_\infty>\inf_{\xi\in\mathbb{R}^3} C(\xi),
\end{equation}
then there exists $\varepsilon^*>0$ such that system ($\mathcal{SP}_{\varepsilon}$) has a positive solution $v_\varepsilon$ for
$\varepsilon\in(0,\varepsilon^*)$.
Moreover, $v_{\varepsilon}$ concentrates at a local (hence global) maximum point $y^*\in\mathbb{R}^3$
such that
$$C(y^*)=\min_{\xi\in\mathbb{R}^3}C(\xi).
$$
\end{theorem}

\begin{remark}
Theorem \ref{principal} complements the study made in \cite{Fang-Zhang,Ianni-Vaira,Yang-Han,Zhao-Liu-Zhao} in the following sense: we deal with
the perturbation problem ($\mathcal{SP}_\varepsilon$) and study the concentration behavior of positive bound states.
\end{remark}

\begin{remark}
To the best of our knowledge, it seems that the only previous paper regarding the concentration of solutions for the perturbed Schrödinger-Poisson
system with potentials $V$ and $K$ is \cite{Ianni-Vaira}, where the smoothness  of such potentials is considered. We only need the boundedness
of $V$ and $K$. Moreover, we do not assume that the concentration point of solutions $v_\varepsilon$ for the system ($\mathcal{SP}_\varepsilon$) is a local minimum (or maximum) of such potentials, as in the previous paper. In our research we shall consider a different variational approach.
\end{remark}

The outline of this paper is as follows: in Section 2 we set the variational framework. In Section 3 we study the autonomous system related
to ($\mathcal{SP}_\varepsilon$). In section 4 we establish an existence result for system ($\mathcal{SP}_\varepsilon$) with $\varepsilon=1$.
In section 5, we prove Theorem \ref{principal}.
\section{Variational framework and preliminary results}

Throughout this paper we use the following notations:
\begin{itemize}
\item $H^1(\mathbb{R}^3)$ is the usual Sobolev space endowed with the standard scalar product and norm
\begin{eqnarray*} (u,v)=\int_{\mathbb{R}^3}(\nabla u\nabla v+uv)\,dx, \quad \|u\|^2=\int_{\mathbb{R}^3}(|\nabla u|^2+u^2)\,dx.
\end{eqnarray*}
\item $\mathcal{D}^{1,2}=\mathcal{D}^{1,2}(\mathbb{R}^3)$ represents the completion of $C_0^\infty(\mathbb{R}^3)$ with respect to the norm
\begin{eqnarray*}\|u\|_{\mathcal{D}^{1,2}}^2=\int_{\mathbb{R}^3}|\nabla u|^2\,dx.
\end{eqnarray*}
\item $L^p(\Omega)$, $1\leq p\leq \infty$, $\Omega\subset\mathbb{R}^3$, denotes a Lebesgue space; the norm in $L^p(\Omega)$ is denoted by
$\|u\|_{L^p(\Omega)}$, where $\Omega $ is a proper subset of $\mathbb{R}^3$; $\|u\|_p$ is the norm in $L^p(\mathbb{R}^3)$.
\end{itemize}

We recall that by the Lax-Milgram theorem, for every $v\in H^1(\mathbb{R}^3)$, the Poisson equation  $-\Delta \phi=v^{2}$  has a unique positive solution
$\phi=\phi_{v}\in \mathcal{D}^{1,2}(\mathbb{R}^3)$ given by
\begin{eqnarray}\label{1} \phi_{v}(x)=\int_{\mathbb{R}^3}\frac{v^{2}(y)}{|x-y|}\,dy.
\end{eqnarray}

The function $\phi: H^{1}(\mathbb{R}^3)\rightarrow \mathcal{D}^{1,2}(\mathbb{R}^3)$, $\phi[v]=\phi_v$ has the following properties
(see for instance Cerami and Vaira \cite{Cerami-Vaira})
\begin{lemma} \label{propriedade phi}For any $v\in H^{1}(\mathbb{R}^3)$, we have
\begin{itemize}\item [i)] $\phi$ is continuous and maps bounded sets into bounded sets;
\item [ii)] $\phi_v\geq 0$;
\item [iii)] there exists $C>0$ such that $\|\phi\|_{D^{1,2}}\leq C\|v\|^2$ and
$$\int_{\mathbb{R}^3}|\nabla v |^2\, dx=\int_{\mathbb{R}^3}\phi_{v} v^2\, dx\leq C\|v\|^4;$$
\item [iv)] $\phi_{tv}=t^2\phi_{v}$, $\forall\,t>0$;
\item [v)] if $v_n\rightharpoonup v$ in $H^1(\mathbb{R}^3)$, then $\phi_{v_n}\rightharpoonup\phi_v$ in $\mathcal{D}^{1,2}(\mathbb{R}^3)$.
\end{itemize}
\end{lemma}

As in \cite{Ambrosetti}, for every $v\in H^1(\mathbb{R}^3)$, there exist a unique solution $\phi=\phi_{K,v}\in \mathcal{D}^{1,2}(\mathbb{R}^3)$ of
$-\Delta \phi=K(x)v^{2}$ where
\begin{eqnarray}\label{1} \phi_{K,v}(x)=\int_{\mathbb{R}^3}\frac{K(y)v^{2}(y)}{|x-y|}dy.
\end{eqnarray}
and it is easy to see that $\phi_{K,v}$ satisfies Lemma \ref{propriedade phi} if $K$ satisfies conditions $(H_0)-(H_1)$.

Substituting (\ref{1}) into the first equation of ($\mathcal{SP}_\varepsilon$), we obtain
\begin{eqnarray}\label{2}
-\varepsilon^2\Delta v+ V(x)v+K(x)\phi_{K,v}(x) v= |v|^{q-2}v.
\end{eqnarray}

Making the changing of variables $x\mapsto \varepsilon x$ and setting $u(x)=v(\varepsilon x)$, (\ref{2}) becomes
\begin{eqnarray}\label{3}
-\Delta u+ V(\varepsilon x)u+K(\varepsilon x)\phi_{K,v}(\varepsilon x) u= |u|^{q-2}u.
\end{eqnarray}

A simple computation shows that
$$\phi_{K,v}(\varepsilon x)=\varepsilon^2 \phi_{\varepsilon,u}(x),
$$
where
$$\phi_{\varepsilon,u}(x)=\int_{\mathbb{R}^3}\frac{K(\varepsilon y)u^2(y)}{|x-y|}dy.
$$

Substituting it into (\ref{3}), ($\mathcal{SP}_\varepsilon$) can be rewritten in the following equivalent equation
\begin{equation}\label{4}\tag{$\mathcal{S}_\varepsilon$}
-\Delta u+ V(\varepsilon x)u+\varepsilon^2 K(\varepsilon x)\phi_{\varepsilon,u} u= |u|^{q-2}u.
\end{equation}

Note that if $u_\varepsilon$ is a solution of (\ref{4}), then $v_\varepsilon(x)=u_\varepsilon(\frac{x}{\varepsilon})$ is a solution of (\ref{2}).

We denote by $H_\varepsilon=\{u\in H^{1}(\mathbb{R}^3):\, \int_{\mathbb{R}^3}V(\varepsilon x)u^2<\infty\}$ the Sobolev space endowed with the norm
\begin{eqnarray*} \|u\|_{\varepsilon}^{2}=\int_{\mathbb{R}^{3}}(|\nabla u|^2+V(\varepsilon x)u^2)\,dx.
\end{eqnarray*}

At this step, we see that ($\mathcal{S}_\varepsilon$) is variational and its solutions are critical points of the functional
\begin{eqnarray*} I_\varepsilon(u)=\frac{1}{2}\int_{\mathbb{R}^3}(|\nabla u|^2+V(\varepsilon x)u^2)\,dx+
\frac{\varepsilon^2}{4}\int_{\mathbb{R}^3}K(\varepsilon x)\phi_{\varepsilon,u}(x) u^2\,dx-\frac{1}{q}\int_{\mathbb{R}^3}|u|^q\,dx.
\end{eqnarray*}
\section{Autonomous Case}

In this section we study the following autonomous system
\begin{equation*} \left \{ \begin{array}{ll}
 -\Delta u+ V(\xi)u+K(\xi)\phi(\xi) u= |u|^{q-2}u & \mbox{in}\quad \mathbb{R}^{3}\\
-\Delta \phi=K(\xi)u^{2} & \mbox{in}\quad \mathbb{R}^{3}
\end{array}\right.\tag{$\mathcal{SP}_{\xi}$}
\end{equation*}
where $\xi\in\mathbb{R}^3$.

We associate with system ($\mathcal{SP}_{\xi}$) the functional $I_\xi: H_\xi\mapsto\mathbb{R}$
\begin{eqnarray}\label{Ixi}I_\xi(u)=\frac{1}{2}\int_{\mathbb{R}^3}(|\nabla u|^2+V(\xi) u^2)\,dx+\frac{1}{4}\int_{\mathbb{R}^3}K(\xi)\phi_u(\xi)
u^2\,dx-\frac{1}{q}\int_{\mathbb{R}^3}|u|^q\,dx.
\end{eqnarray}

Hereafter, the Sobolev space $H_\xi=H^1(\mathbb{R}^3)$ is endowed with the norm
\begin{eqnarray*}\|u\|_\xi=\int_{\mathbb{R}^3}(|\nabla u|^2 +V(\xi) u^2)\,dx.
\end{eqnarray*}

By standard arguments, the functional $I_\xi$ verifies the Mountain-Pass Geometry, more exactly it satisfies the following lemma
\begin{lemma}\label{mountain pass geometry} The functional $I_\xi$ satisfies
\begin{itemize}
 \item [(i)] There exist positive constants $\alpha,\rho$ such that $I_\xi(u)\geq\alpha$ for  $\|u\|_\xi=\rho$.
 \item [(ii)] There exists $u_{1}\in H^1(\mathbb{R}^3)$ with $\|u_{1}\|_\xi>\rho$ such that $I_\xi(u_{1})< 0$.
\end{itemize}
\end{lemma}

Applying a variant of the Mountain Pass Theorem (see \cite{Willem}), we obtain a sequence $(u_n)\subset H^1(\mathbb{R}^3)$ such that
\begin{eqnarray*}I_\xi(u_{n})\rightarrow C(\xi) \hspace{0.2cm}\text{and}\hspace{0.2cm} I_\xi'(u_{n})\rightarrow 0,
\end{eqnarray*}
where
\begin{eqnarray}\label{valorC}C(\xi)=\inf_{\gamma\in\Gamma}\max_{0\leq t\leq 1}I_\xi(\gamma(t)), \hspace{0.2cm} C(\xi)\geq\alpha
\end{eqnarray}
\noindent and
\begin{eqnarray*}\Gamma=\{\gamma\in\mathcal{C}([0,1],H^1(\mathbb{R}^3) )|\gamma(0)=0, \gamma(1)=u_{1} \}.
\end{eqnarray*}

We observe that $C(\xi)$ can be also characterized as
\begin{eqnarray*}C(\xi)=\inf_{u\neq 0}\max_{t>0}I_\xi(tu).
\end{eqnarray*}

\begin{proposition}\label{prop1} Let $\xi\in\mathbb{R}^3$. Then system ($\mathcal{SP}_{\xi}$) has a positive solution $u\in H^1(\mathbb{R}^3)$ such that
$I'_\xi(u)=0$  and $I_\xi(u)=C(\xi)$, for any $q\in(4,6)$.
\end{proposition}

\begin{proof}
The proof is an easy adaptation of Theorem 1.1 in \cite{Azzollini-Pomponio-SM} and we omit it.
\end{proof}

\begin{lemma} The function $\xi\mapsto C(\xi)$ is continuous.
\end{lemma}

\begin{proof} The proof consists in proving that there exist sequences $(\zeta_n)$ and $(\lambda_n)$ in $\mathbb{R}^3$ such that
$C(\zeta_n), C(\lambda_n)\rightarrow C(\xi)$ as $n\rightarrow 0$ $\zeta_n\rightarrow\xi$, where
\begin{itemize}
\item []$\zeta_n\rightarrow\xi $ and $C(\zeta_n)\geq C(\xi)$, $\forall \,n$
\item []$\lambda_n\rightarrow \xi$ and $C(\lambda_n)\geq C(\xi)$, $\forall \,n$
\end{itemize}
as we know by Alves and Soares \cite{Alves-Soares} with slightly modifications.

\end{proof}

\begin{remark} The function $(\mu,\nu)\mapsto c_{\mu,\nu}$ is continuous, where $c_{\mu,\nu}$ is the minimax level of
\begin{eqnarray}\label{Imu}I_{\mu,\nu}(u)=\frac{1}{2}\int_{\mathbb{R}^3}(|\nabla u|^2+\mu u^2)\,dx+\frac{1}{4}\int_{\mathbb{R}^3}\nu\phi_u(x)
u^2\,dx-\frac{1}{q}\int_{\mathbb{R}^3}|u|^q\,dx.
\end{eqnarray}
\end{remark}

\begin{remark}  We denote by $C_\infty$ the minimax value related to the functional
\begin{eqnarray*}I_\infty(u)=\frac{1}{2}\int_{\mathbb{R}^3}(|\nabla u|^2+V_\infty u^2)\,dx+\frac{1}{4}\int_{\mathbb{R}^3}K_\infty\phi_u
u^2\,dx-\frac{1}{q}\int_{\mathbb{R}^3}|u|^q\,dx
\end{eqnarray*}
where $V_\infty$ and $K_\infty$, given by condition $(H_1)$, belong to $(0,\infty)$. Otherwise, define $C_\infty=\infty$.
$I_\infty(u)$ is well defined for $u\in H_\infty$, where $H_\infty$ is a Sobolev space endowed with the norm
$$\|u\|_\infty=\int_{\mathbb{R}^3}(|\nabla u|^2 +V_\infty u^2)\,dx
$$
equivalent to the usual Sobolev norm on $H^1(\mathbb{R}^3)$.
\end{remark}
\section{System $(\mathcal{S}_{1})$}\label{S-1}

Setting $\varepsilon=1$, in this section we consider the following system
\begin{equation*} \left \{ \begin{array}{ll}
-\Delta u+ V(x)u+K(x)\phi(x) u= |u|^{q-2}u & \mbox{in}\quad \mathbb{R}^{3}\\
-\Delta \phi=K(x)u^{2} & \mbox{in}\quad \mathbb{R}^{3}
\end{array}\right.\tag{$\mathcal{SP}_1$}
\end{equation*}
whose solutions are critical points of the corresponding functional
\begin{eqnarray*} I(u)=\frac{1}{2}\int_{\mathbb{R}^3}(|\nabla u|^2+V( x)u^2)\,dx+ \frac{1}{4}\int_{\mathbb{R}^3}K(x)\phi_{u}(x)
u^2\,dx-\frac{1}{p}\int_{\mathbb{R}^3}|u|^p\,dx
\end{eqnarray*}
which is well defined for $u\in H_1$, where
\begin{eqnarray*}H_1=\{u\in H^{1}(\mathbb{R}^3):\, \int_{\mathbb{R}^3}V(x)u^2\,dx<\infty\}
\end{eqnarray*}
with the same norm notation of the Sobolev space $H^1(\mathbb{R}^3)$.

Similar to the autonomous case, the functional $I$ satisfies the mountain pass geometry, then there exists a sequence $(u_n)\subset H_1$ such that
\begin{eqnarray}\label{5}I(u_n)\rightarrow c \quad \mbox{and} \quad I'(u_n)\rightarrow 0
\end{eqnarray}
where
\begin{eqnarray*}c=\inf_{\gamma\in\Gamma}\max_{0\leq t\leq 1}I(\gamma(t))
\end{eqnarray*}
\noindent and
\begin{eqnarray*}\Gamma=\{\gamma\in\mathcal{C}([0,1],H_{1}(\mathbb{R}^{3}))|\gamma(0)=0, I(\gamma(1))<0 \}.
\end{eqnarray*}

An important tool in our analysis is the following theorem:

\begin{theorem}\label{critical value}
If $c<C_{\infty}$, then $c$ is a critical value for $I$.
\end{theorem}

\begin{proof} From (\ref{5}), $(u_n)$ is bounded in $H_1$. As a consequence, passing to a subsequence if necessary, $u_n\rightharpoonup u$ in $H_1$.
From Proposition \ref{propriedade phi} (v), $\phi_{u_n}\rightharpoonup \phi_u$ in $\mathcal{D}^{1,2}(\mathbb{R}^3)$, as $n\rightarrow\infty$. Then,
$(u, \phi_u)$ is a weak solution of ($\mathcal{SP}_1$). Similar to the proof of Lemma \ref{mountain pass geometry}, $I(u)=c$.
It remains to show that $u\neq 0$.

By contradiction, consider $u\equiv 0$.

From Alves, Souto and Soares \cite{Alves-Souto-Soares}, if there exist constants $\eta$, $R$ such that
\begin{eqnarray*}\liminf_{n\rightarrow+\infty}\int_{B_R(0)}u_n^2\,dx\geq\eta>0
\end{eqnarray*}
then $u\neq 0$.

Hence, there exists a subsequence of $(u_n)$, still denoted by $(u_n)$, such that
\begin{eqnarray*}\lim_{n\rightarrow+\infty}\int_{B_R(0)}u_n^2\,dx=0.
\end{eqnarray*}

Let $\mu$ and $\nu$ be such that
\begin{eqnarray*}
\inf_{x\in\mathbb{R}^3}V(x)<\mu<\liminf_{|x|\rightarrow\infty}V(x)=V_{\infty}\\
\inf_{x\in\mathbb{R}^3}K(x)<\nu<\liminf_{|x|\rightarrow\infty}K(x)=K_{\infty}
\end{eqnarray*}
and take $R>0$ such that
\begin{eqnarray*}
V(x)>\mu, \quad \forall\,x\in \mathbb{R}^3\backslash B_R(0)\\
K(x)>\nu, \quad \forall\,x\in \mathbb{R}^3\backslash B_R(0).
\end{eqnarray*}

For each $n\in \mathbb{N}$, there exist $t_n>0$, $t_n\rightarrow 1$ such that $\displaystyle I(t_n u_n)=\max_{t\geq 0}I(t u_n)$.
The convergence of $(t_n)$ follows from (\ref{5}). In fact, since
\begin{eqnarray*} \|u_n\|^2+\int_{\mathbb{R}^3}K(x)\phi_{u_n} u_n^2\,dx=\int_{\mathbb{R}^3}|u_n|^q\,dx+o_n(1)
\end{eqnarray*}
we have
\begin{eqnarray*} t_n^2\|u_n\|^2+t_n^4\int_{\mathbb{R}^3}K(x)\phi_{u_n} u_n^2\,dx=t_n^q\int_{\mathbb{R}^3}|u_n|^q\,dx+o_n(1).
\end{eqnarray*}

Then,
\begin{eqnarray*} (1-t_n^2)\|u_n\|^2=(t_n^{q-2}-t_n^2)\int_{\mathbb{R}^3}|u_n|^q\,dx + o_n(1)
\end{eqnarray*}
Suppose $t_n\rightarrow t_0$. Letting $n\rightarrow+\infty$,
\begin{eqnarray*} 0=(t_0^2-1)\ell_1+t_0^2(t_0^{q-4}-1)\ell_2
\end{eqnarray*}
where $\ell_1,\ell_2>0$. Hence, $t_0=1$.

Consequently, we have
\begin{eqnarray*}I(u_n)-I(t_n u_n)&=&\frac{1-t_n^2}{2}\|u_n\|^2+\frac{1}{4}(1-t_n^4)\int_{\mathbb{R}^3}K(x)\phi_{u_n} u_n^2\,dx+
\frac{t_n^q-1}{q}\int_{\mathbb{R}^3}|u_n|^q\,dx\\
&=& o_n(1)
\end{eqnarray*}
which implies, for every $t\geq 0$,
\begin{eqnarray}\label{bla}
I(u_n)&\geq & I(t u_n)+o_n(1)\nonumber\\
&=& \frac{t^2}{2}\int_{\mathbb{R}^3}|\nabla u_n|^2+V(x)u_n^2\,dx + \frac{t^4}{4}\int_{\mathbb{R}^3}K(x)\phi_{u_n} u_n^2\,dx-
\frac{t^q}{q}\int_{\mathbb{R}^3}|u_n|^q\,dx+\nonumber\\
&&+I_{\mu,\nu}(t u_n)-I_{\mu,\nu}(t u_n)+o_n(1)\nonumber\\
&\geq & \frac{t^2}{2}\int_{B_R(0)}(V(x)-\mu)u_n^2\,dx +\frac{t^4}{4}\int_{B_R(0)}(K(x)-\nu)\phi_{u_n} u_n^2\,dx+ \nonumber \\
&& +I_{\mu,\nu}(t u_n)+o_n(1),
\end{eqnarray}
where $I_{\mu,\nu}(u)$ is given by (\ref{Imu}).

Consider $\tau_n$ such that $\displaystyle I_{\mu,\nu}(\tau_n u_n)=\max_{t\geq 0}I_{\mu,\nu}(tu_n)$. As in the above arguments, $\tau_n\rightarrow 1$.
Letting $t=\tau_n$ in (\ref{bla}), we have
\begin{eqnarray*} I(u_n)\geq c_{\mu,\nu} + \frac{\tau_n^2}{2}\int_{B_R(0)}(V(x)-\mu)u_n^2\,dx+
\frac{\tau_n^4}{4}\int_{B_R(0)}(K(x)-\nu)\phi_{u_n} u_n^2\,dx+o_n(1).
\end{eqnarray*}

Taking the limit $n\rightarrow +\infty$, we have $c\geq c_{\mu,\nu}$. Next, taking $\mu\rightarrow V_\infty$ and $\nu\rightarrow K_\infty$, we obtain
$c\geq C_\infty$, proving Theorem \ref{critical value}.

\end{proof}
\section{Proof of Theorem \ref{principal}}

This section is devoted to study the existence, regularity and the asymptotic behavior of solutions for the system ($\mathcal{SP}_{\varepsilon}$) for small
$\varepsilon$. The proof of Theorem \ref{principal} is divided into three subsections as follows:

\subsection{Existence of a solution}

\begin{theorem} Suppose $(H_0)-(H_1)$ hold and consider
\begin{equation}\tag{$C^\infty$} C_\infty>\inf_{\xi\in\mathbb{R}^3} C(\xi)
\end{equation}
Then, there exists $\varepsilon^*>0$ such that system ($\mathcal{S}_{\varepsilon}$) has a positive solution for every
$0<\varepsilon<\varepsilon^*$.
\end{theorem}

\begin{proof}
By hypothesis ($C^\infty$), there exists $b\in\mathbb{R}^3$ and $\delta>0$ such that
\begin{eqnarray}\label{8}C(b)+\delta<C_{\infty}.
\end{eqnarray}

Define $u_\varepsilon(x)=u(x-\frac{b}{\varepsilon})$, where, from Proposition \ref{prop1}, $u$ is a solution of the autonomous
Schrödinger-Poisson system ($\mathcal{SP}_b$)
\begin{eqnarray*} \left \{ \begin{array}{ll}
 -\Delta u+ V(b)u+K(b)\phi(b) u= |u|^{q-2}u & \mbox{in}\quad \mathbb{R}^{3}\\
-\Delta \phi=K(b)u^{2} & \mbox{in}\quad \mathbb{R}^{3}
\end{array}\right.
\end{eqnarray*}
with $I_b(u)=C(b)$.

Let $t_\varepsilon$ be such that $\displaystyle I_\varepsilon(t_\varepsilon u_\varepsilon)=\max_{t\geq 0}I_\varepsilon(t u_\varepsilon)$.
Similar to the proof of Theorem \ref{critical value}, we have $\displaystyle \lim_{\varepsilon\rightarrow 0}t_\varepsilon=1$.

Then, since
\begin{eqnarray*}
c_\varepsilon= \inf_{\gamma\in\Gamma}\max_{0\leq t\leq 1}I_\varepsilon(\gamma(t))=\inf_{\genfrac{}{}{0pt}{}{u\in H^1}{u\neq 0} }
\max_{ t\geq 0}I_\varepsilon(tu)\leq
\max_{ t\geq 0}I_\varepsilon(t u_\varepsilon)=I_\varepsilon(t_\varepsilon u_\varepsilon)
\end{eqnarray*}
we have
\begin{eqnarray*}
\limsup_{\varepsilon\rightarrow 0} c_\varepsilon\leq \limsup_{\varepsilon\rightarrow 0} I_\varepsilon(t_\varepsilon u_\varepsilon)
=I_b(u)=C(b)<C(b)+\delta,
\end{eqnarray*}
which implies that, from (\ref{8})
\begin{eqnarray*}\limsup_{\varepsilon\rightarrow 0} c_\varepsilon< C_\infty.
\end{eqnarray*}

Therefore, there exists $\varepsilon^*>0$ such that $c_\varepsilon<C_\infty $ for every $0<\varepsilon<\varepsilon^*$. In view of Theorem
\ref{critical value}, system $(\mathcal{S}_\varepsilon)$ has a positive solution for every $0<\varepsilon<\varepsilon^*$.
\end{proof}

\subsection{Regularity of the solution}

The first result is a suitable version of Brezis and Kato \cite{Brezis-Kato} and the second one is a particular version of Theorem 8.17 from Gilbarg and
 Trudinger \cite{Gilbarg-Trudinger}.

\newpage

\begin{proposition}\label{R1}
Consider $u\in H^1(\mathbb{R}^3)$ satisfying
\begin{eqnarray*}-\Delta u+b(x)u=f(x,u)\quad \mbox{in}\,\, \mathbb{R}^3
\end{eqnarray*}
where $b: \mathbb{R}^3\rightarrow \mathbb{R}$ is a $L^\infty_{loc}(\mathbb{R}^3)$ function and $f:\mathbb{R}^3\rightarrow \mathbb{R}$
is a Caratheodory function such that
\begin{eqnarray*}0\leq f(x,s)\leq C_f(s^r+s), \quad \forall \, s>0,\, x\in\mathbb{R}^3.
\end{eqnarray*}

Then, $u\in L^t(\mathbb{R}^3)$ for every $t\geq 2$. Moreover, there exists a positive constant $C=C(t,C_f)$ such that
\begin{eqnarray*}\|u\|_{L^t(\mathbb{R}^3)}\leq C\|u\|_{H^1(\mathbb{R}^3)}.
\end{eqnarray*}
\end{proposition}

\begin{proposition}\label{R2}
Consider $t>3$ and $g\in L^{\frac{t}{2}}(\Omega)$, where $\Omega$ is an open subset of $\mathbb{R}^3$. Then, if $u\in H^1(\Omega)$ is a subsolution
of
\begin{eqnarray*} \Delta u=g \quad \mbox{in} \,\,\Omega
\end{eqnarray*}
we have, for any $y\in\mathbb{R}^3$ and $B_{2R}(y)\subset \Omega$, $R>0$
\begin{eqnarray*}\sup_{B_{R}(y)}u \leq C \Big(\|u^+\|_{L^2(B_{2R}(y))}+\|g\|_{L^{\frac{t}{2}}(B_{2R}(y)) }\Big)
\end{eqnarray*}
where $C=C(t,R)$.
\end{proposition}

In view of Propositions \ref{R1} and \ref{R2}, the positive solutions of $(\mathcal{SP}_{\varepsilon})$ are in
$C^2(\mathbb{R}^3)\cap L^{\infty}(\mathbb{R}^3)$ for all $\varepsilon >0$. Similar arguments was employed by He and Zou \cite{He-Zou2012}.

\subsection{Concentration of solutions}

\begin{lemma}\label{beta0} Suppose $(H_0)-(H_1)$ hold. Then, there exists $\beta_0>0$ such that
\begin{eqnarray*}c_\varepsilon\geq \beta_0,
\end{eqnarray*}
for every $\varepsilon>0$. Moreover,
\begin{eqnarray*}\limsup_{\varepsilon\rightarrow 0}c_\varepsilon\leq\inf_{\xi\in\mathbb{R}^3}C(\xi).
\end{eqnarray*}
\end{lemma}

\begin{proof}
Let $w_\varepsilon\in H_\varepsilon$ be such that $c_\varepsilon=I_\varepsilon(w_\varepsilon)$. Then, from condition $(H_0)$
\begin{eqnarray*}
c_\varepsilon=I_\varepsilon(w_\varepsilon)\geq \inf_{\genfrac{}{}{0pt}{}{u\in H^1}{u\neq 0} } \sup_{t\geq 0}J(tu)=\beta_0,\,\,\,  \forall \varepsilon>0,
\end{eqnarray*}
where
\begin{eqnarray*}
J(u)=\frac{1}{2}\int_{\mathbb{R}^3}(|\nabla u|^2+\alpha u^2)\,dx+\frac{1}{4}\int_{\mathbb{R}^3}\alpha\phi_u
u^2\,dx-\frac{1}{q}\int_{\mathbb{R}^3}|u|^q\,dx.
\end{eqnarray*}

Let $\xi\in\mathbb{R}^3$ and consider $w\in H^1(\mathbb{R}^3)$ a least energy solution for system $(\mathcal{SP}_\xi)$, that is, $I_\xi(w)=C(\xi)$ and
$I'_{\xi}(w)=0$. Let $w_\varepsilon(x)=w(x-\frac{\xi}{\varepsilon})$ and take $t_\varepsilon>0$ such that
\begin{eqnarray*} c_\varepsilon\leq I_\varepsilon(t_\varepsilon w_\varepsilon)=\max_{t\geq 0}I_\varepsilon(t w_\varepsilon).
\end{eqnarray*}

Similar to the proof of Theorem \ref{critical value}, $t_\varepsilon\rightarrow 1$ as $\varepsilon\rightarrow 0$, then
\begin{eqnarray*} c_\varepsilon\leq I_\varepsilon(t_\varepsilon w_\varepsilon)\rightarrow I_\xi(w)=C(\xi), \quad \mbox{as}\,\, \varepsilon\rightarrow 0
\end{eqnarray*}
which implies that $\displaystyle\limsup_{\varepsilon\rightarrow 0}c_{\varepsilon}\leq C(\xi), \,\,\, \forall \,\, \xi\in\mathbb{R}^3$.

Therefore,
\begin{eqnarray*} \limsup_{\varepsilon\rightarrow 0}c_{\varepsilon}\leq \inf_{\xi\in\mathbb{R}^3} C(\xi).
\end{eqnarray*}

\end{proof}

\begin{lemma}\label{awayfromzero} There exist a family $(y_\varepsilon)\subset\mathbb{R}^3$ and constants $R,\beta>0$ such that
\begin{eqnarray*}\liminf_{\varepsilon\rightarrow 0}
\int_{B_R(y_\varepsilon)}u_\varepsilon^2\, dx \geq \beta, \quad \mbox{for each}\,\,\varepsilon>0.
\end{eqnarray*}
\end{lemma}

\begin{proof}
By contradiction, suppose that there exists a sequence $\varepsilon_n\rightarrow 0$ such that
\begin{eqnarray*}\lim_{n\rightarrow\infty} \sup_{y\in\mathbb{R}^3} \int_{B_R(y)}u_n^2\, dx = \beta, \quad \mbox{for all}\,\,R>0.
\end{eqnarray*}
where, for the sake of simplicity, we denote $u_n(x)=u_{\varepsilon_n}(x)$. Hereafter, denote $\phi_{\varepsilon_n,u_n}(x)=\phi_{u_n}(x)$.

From Lemma I.1 in \cite{Lions2}, we have
\begin{eqnarray*}\int_{\mathbb{R}^3}|u_n|^q\,dx\rightarrow 0, \quad \mbox{as}\,\, n\rightarrow\infty.
\end{eqnarray*}

But, since,
\begin{eqnarray*}
\int_{\mathbb{R}^3}(|\nabla u_n|^2+V(\varepsilon_n x)u_n^2)\,dx+ \int_{\mathbb{R}^3}\varepsilon_n^2 K(\varepsilon_n x)\phi_{u_n}
u_n^2\,dx=\int_{\mathbb{R}^3}|u_n|^q\,dx
\end{eqnarray*}
we have
\begin{eqnarray*}
\int_{\mathbb{R}^3}(|\nabla u_n|^2+V(\varepsilon_n x)u_n^2)\,dx\rightarrow 0,  \quad \mbox{as}\,\, n\rightarrow\infty.
\end{eqnarray*}

Therefore,
\begin{eqnarray*}
\lim_{n\rightarrow\infty}c_{\varepsilon_n}=\lim_{n\rightarrow\infty}I_{\varepsilon_n}(u_n)=0
\end{eqnarray*}
which is an absurd, since for some $\beta_0>0$, $c_\varepsilon\geq\beta_0$, from Lemma \ref{beta0}.
\end{proof}

\begin{lemma}\label{lema1} The family $(\varepsilon y_\varepsilon)$ is bounded. Moreover, if $y^*$ is the limit of the sequence
$(\varepsilon_n  y_{\varepsilon_n})$ in the family $(\varepsilon y_\varepsilon)$, then we have
\begin{eqnarray*}C(y^*)=\inf_{\xi\in\mathbb{R}^3}C(\xi).
\end{eqnarray*}
\end{lemma}

\begin{proof}
Consider $u_n(x)=u_{\varepsilon_n}(x+y_{\varepsilon_n})$. Suppose by contradiction that $(\varepsilon_n y_{\varepsilon_n})$ goes to infinity.

It follows from Lemma \ref{awayfromzero} that there exists constants $R,\beta>0$ such that
\begin{eqnarray}\label{11}\int_{B_R(0)}u_n^2(x)\, dx \geq \beta>0, \quad \mbox{for all}\,\,n\in\mathbb{N}.
\end{eqnarray}

Since $u_n(x)$ satisfies
\begin{eqnarray}\label{14}-\Delta u_n+ V(\varepsilon_n x+\varepsilon_n y_{\varepsilon_n})u_n+\varepsilon_n^2 K(\varepsilon_n x+\varepsilon_n
y_{\varepsilon_n})\phi_{\varepsilon_n,u_n} u_n= |u_n|^{q-2}u_n,
\end{eqnarray}
then, $u_n(x)$ is bounded in $H_{\varepsilon}$. Hence, passing to a subsequence if necessary, $u_n\rightarrow \hat{u}
\geq 0$ weakly in $H_{\varepsilon}$, strongly in $L_{loc}^p(\mathbb{R}^3)$ for $p\in(2,6)$ and a.e. in $\mathbb{R}^3$.
From (\ref{11}), $\hat{u}\neq 0$.

Using $\hat{u}$ as a test function in (\ref{14}) and taking the limit, we get
\begin{eqnarray}\label{15}
\int_{\mathbb{R}^3}(|\nabla \hat{u}|^2+\mu\hat{u}^2)\,dx \leq  \int_{\mathbb{R}^3}(|\nabla \hat{u}|^2+\mu\hat{u}^2)\,dx +
 \int_{\mathbb{R}^3}\nu \phi_{\hat{u}} \hat{u}^2\,dx \leq \int_{\mathbb{R}^3}|\hat{u}|^q\,dx
\end{eqnarray}
where, $\mu$ and $\nu$ are positive constantes such that
\begin{eqnarray*}\mu<\liminf_{|x|\rightarrow\infty}V(x) \quad \mbox{and}\quad \nu<\liminf_{|x|\rightarrow\infty}K(x).
\end{eqnarray*}

Consider the functional $I_{\mu,\nu}:H^1(\mathbb{R}^3)\rightarrow \mathbb{R}$ given by
\begin{eqnarray*}I_{\mu,\nu}(u)=\frac{1}{2}\int_{\mathbb{R}^3}(|\nabla u|^2+\mu u^2)\,dx+\frac{1}{4}\int_{\mathbb{R}^3}\nu\phi_u(x)
u^2\,dx-\frac{1}{q}\int_{\mathbb{R}^3}|u|^q\,dx.
\end{eqnarray*}

Let $\sigma>0$ be such that  $\displaystyle I_{\mu,\nu}(\sigma \hat{u})=\max_{t>0}I_{\mu,\nu}(t \hat{u})$.

We claim that
\begin{eqnarray}\label{12}
\sigma^2 \int_{\mathbb{R}^3}(|\nabla \hat{u}|^2+\mu\hat{u}^2)\,dx+\sigma^4 \int_{\mathbb{R}^3}\nu \phi_{\hat{u}} \hat{u}^2\,dx=
\sigma^q\int_{\mathbb{R}^3}|\hat{u}|^q\,dx.
\end{eqnarray}

In fact, from (\ref{15})
\begin{eqnarray*}
I_{\mu,\nu}(\sigma \hat{u})&=&
\frac{\sigma^2}{2} \int_{\mathbb{R}^3}(|\nabla \hat{u}|^2+\mu\hat{u}^2)\,dx+\frac{\sigma^4}{4} \int_{\mathbb{R}^3}\nu \phi_{\hat{u}} \hat{u}^2\,dx-
\frac{\sigma^q}{q}\int_{\mathbb{R}^3}|\hat{u}|^q\,dx\\
& \leq &
\frac{\sigma^2}{2} \int_{\mathbb{R}^3}|\hat{u}|^q\,dx+\frac{\sigma^4}{4} \int_{\mathbb{R}^3}\nu \phi_{\hat{u}} \hat{u}^2\,dx-
\frac{\sigma^q}{q}\int_{\mathbb{R}^3}|\hat{u}|^q\,dx
\end{eqnarray*}
it follows that $\sigma\leq 1$, and since $\frac{d}{dt}I_{\mu,\nu}(t\hat{u})\Big|_{t=\sigma}=0$, we obtain
\begin{eqnarray*}
\frac{d}{dt}I_{\mu,\nu}(t \hat{u})\Big|_{t=\sigma}=
\sigma \int_{\mathbb{R}^3}(|\nabla \hat{u}|^2+\mu\hat{u}^2)\,dx+\sigma^3 \int_{\mathbb{R}^3}\nu \phi_{\hat{u}} \hat{u}^2\,dx-
\sigma^{q-1}\int_{\mathbb{R}^3}|\hat{u}|^q\,dx=0
\end{eqnarray*}
proving (\ref{12}).

From Lemma \ref{beta0}, equation (\ref{12}) and the fact that $\sigma\leq 1$, we have
\begin{eqnarray*}
c_{\mu,\nu} &=& \inf_{u\neq 0}\max_{t>0}I_{\mu,\nu}(tu)=\inf_{u\neq 0}I_{\mu,\nu}(\sigma u)\leq I_{\mu,\nu}(\sigma \hat{u})\\
&=&
\frac{\sigma^2}{2} \int_{\mathbb{R}^3}(|\nabla \hat{u}|^2+\mu\hat{u}^2)\,dx+\frac{\sigma^4}{4} \int_{\mathbb{R}^3}\nu \phi_{\hat{u}} \hat{u}^2\,dx-
\frac{\sigma^q}{q}\int_{\mathbb{R}^3}|\hat{u}|^q\,dx\\
&= &
\frac{\sigma^2}{4} \int_{\mathbb{R}^3}(|\nabla \hat{u}|^2+\mu\hat{u}^2)\,dx+ \frac{\sigma^q(q-4)}{4q}\int_{\mathbb{R}^3}|\hat{u}|^q\,dx\\
&\leq &
\frac{1}{4} \int_{\mathbb{R}^3}(|\nabla \hat{u}|^2+\mu\hat{u}^2)\,dx+ \frac{q-4}{4q}\int_{\mathbb{R}^3}|\hat{u}|^q\,dx\\
&\leq &
\liminf_{n\rightarrow\infty}\Big(I_{\varepsilon_n}(u_n)-\frac{1}{4}I'_{\varepsilon_n}(u_n)u_n\Big)\\
&=&
\liminf_{n\rightarrow\infty}c_{\varepsilon_n}\leq \limsup_{n\rightarrow\infty}c_{\varepsilon_n}\leq \inf_{\xi\in\mathbb{R}^3}C(\xi)
\end{eqnarray*}
hence, $c_{\mu,\nu}\leq \inf_{\xi\in\mathbb{R}^3}C(\xi) $.

If we consider
\begin{eqnarray*}\mu\rightarrow \liminf_{|x|\rightarrow\infty}V(x)=V_{\infty} \quad \mbox{and}\quad
\nu\rightarrow \liminf_{|x|\rightarrow\infty}K(x)=K_{\infty}
\end{eqnarray*}
then, by the continuity of the function $(\mu,\nu)\mapsto c_{\mu\nu}$ we obtain $\displaystyle C_\infty\leq \inf_{\xi\in\mathbb{R}^3}C(\xi)$,
which contradicts condition
$(C^\infty)$. Therefore, $(\varepsilon y_\varepsilon)$ is bounded and there exists a subsequence of $(\varepsilon y_\varepsilon)$ such that $\varepsilon_n
 y_{\varepsilon_n}\rightarrow
 y^*$.

Now we proceed to prove that $\displaystyle C(y^*)=\inf_{\xi\in\mathbb{R}^3}C(\xi)$.

Recalling that
$u_n(x)=u_{\varepsilon_n}(x+y_{\varepsilon_n})$ and from the arguments above, $\hat{u}$ satisfies the equation
\begin{eqnarray} \label{13}-\Delta u+V(y^*)u+K(y^*)\phi_u u=|u|^{q-2}u
\end{eqnarray}

The Euler-Lagrange functional associated to this equations is  $I_{y^*}: H_{y^*}(\mathbb{R}^3)$, defined as in (\ref{Ixi}) with $\xi=y^*$.

Using $\hat{u}$ as a test function in (\ref{13}) and taking the limit, we obtain
\begin{eqnarray*}
\int_{\mathbb{R}^3}(|\nabla \hat{u}|^2+V(y^*)\hat{u}^2)\,dx \leq \int_{\mathbb{R}^3}|\hat{u}|^q\,dx.
\end{eqnarray*}
Then,
\begin{eqnarray*} I_{y^*}(\sigma \hat{u})=\max_{t>0}I_{y^*}(t\hat{u}).
\end{eqnarray*}

Finally, from Lemma \ref{beta0} and since $0<\sigma\leq 1$ we have
\begin{eqnarray*}
\inf_{\xi\in\mathbb{R}^3}C(\xi) &\leq & C(y^*) \leq I_{y^*}(\sigma \hat{u}) \\
& = &
\frac{\sigma^2}{4} \int_{\mathbb{R}^3}(|\nabla \hat{u}|^2+V(y^*)\hat{u}^2)\,dx+ \frac{\sigma^q(q-4)}{4q}\int_{\mathbb{R}^3}|\hat{u}|^q\,dx\\
&\leq &
\frac{1}{4} \int_{\mathbb{R}^3}(|\nabla \hat{u}|^2+V(y^*)\hat{u}^2)\,dx+ \frac{q-4}{4q}\int_{\mathbb{R}^3}|\hat{u}|^q\,dx\\
&\leq &
\liminf_{n\rightarrow\infty} \Big[ \frac{1}{4} \int_{\mathbb{R}^3}\Big(|\nabla u_n|^2+V(\varepsilon_n x+\varepsilon_n y_{\varepsilon_n})u_n^2\Big)\,dx+
 \frac{q-4}{4q}\int_{\mathbb{R}^3}|u_n|^q\,dx\Big]\\
&\leq &
\liminf_{n\rightarrow\infty}\Big(I_{\varepsilon_n}(u_n)-\frac{1}{4}I'_{\varepsilon_n}(u_n)u_n\Big)\\
&=&
\liminf_{n\rightarrow\infty}c_{\varepsilon_n}\leq \inf_{\xi\in\mathbb{R}^3}C(\xi)
\end{eqnarray*}
which implies that $\displaystyle C(y^*)=\inf_{\xi\in\mathbb{R}^3}C(\xi)$.
\end{proof}

As a consequence of the previous lemma, there exists a subsequence of $(\varepsilon_n y_{\varepsilon_n})$ such that $\varepsilon_n y_{\varepsilon_n}
\rightarrow y^*$.

Let $u_{\varepsilon_n}(x+y_{\varepsilon_n})=u_n(x)$ and consider $\tilde{u}\in H^1$ such that $u_n\rightharpoonup \tilde{u}$.

\begin{lemma} $u_n\rightarrow \tilde{u}$ in $H^{1}(\mathbb{R}^3)$, as $n\rightarrow\infty$. Moreover, there exists $\varepsilon^*>0$ such that
$\lim_{|x|\rightarrow\infty}u_\varepsilon(x)=0$ uniformly on $\varepsilon\in (0,\varepsilon^*)$.
\end{lemma}

\begin{proof}
By applying Lemmas \ref{beta0} and \ref{lema1}, we observe that
\begin{eqnarray*}
\inf_{\xi\in\mathbb{R}^3}C(\xi) &=& C(y^*)\leq I_{y^*}(\tilde{u})-\frac{1}{4}I'_{y^*}(\tilde{u})\tilde{u}\\
&=&\frac{1}{4}\int_{\mathbb{R}^3}(|\nabla \tilde{u}|^2+V(y^*)\tilde{u}^2)\,dx+
\Big(\frac{q-4}{4q}\Big)\int_{\mathbb{R}^3}|\tilde{u}|^q\,dx\\
&\leq & \liminf_{n\rightarrow\infty}\frac{1}{4}\int_{\mathbb{R}^3}(|\nabla u_n|^2+V(\varepsilon_n x+\varepsilon_n y_{\varepsilon_n})u_n^2)\,dx+
\Big(\frac{q-4}{4q}\Big)\int_{\mathbb{R}^3}|u_n|^q\,dx\\
&\leq & \limsup_{n\rightarrow\infty}\frac{1}{4}\int_{\mathbb{R}^3}(|\nabla u_n|^2+V(\varepsilon_n x+\varepsilon_n
y_{\varepsilon_n})u_n^2)\,dx+\Big(\frac{q-4}{4q}\Big)\int_{\mathbb{R}^3}|u_n|^q\,dx\\
&=& \limsup_{n\rightarrow\infty}\Big( I_{\varepsilon_n}(u_{\varepsilon_n})-\frac{1}{4}I'_{\varepsilon_n}(u_{\varepsilon_n})u_{\varepsilon_n}\Big) \\
& =& \limsup_{n\rightarrow\infty} c_{\varepsilon_n} \leq  \inf_{\xi\in\mathbb{R}^3}C(\xi)
\end{eqnarray*}
then,
\begin{eqnarray*}
\lim_{n\rightarrow\infty}\int_{\mathbb{R}^3}(|\nabla u_n|^2+V(\varepsilon_n x+\varepsilon_n y_{\varepsilon_n})u_n^2)\,dx=
\int_{\mathbb{R}^3}(|\nabla \tilde{u}|^2+V(y^*)\tilde{u}^2)\,dx.
\end{eqnarray*}

Now observe that
\begin{eqnarray*}
c_{\varepsilon_n} &=& I_{\varepsilon_n}(u_{\varepsilon_n})-\frac{1}{4}I'_{\varepsilon_n}(u_{\varepsilon_n})u_{\varepsilon_n}\\
&=& \frac{1}{4}\int_{\mathbb{R}^3}(|\nabla u_{\varepsilon_n}|^2+V(\varepsilon_n x)u_{\varepsilon_n}^2)\,dx+
\Big(\frac{q-4}{4q}\Big)\int_{\mathbb{R}^3}|u_{\varepsilon_n}|^q\,dx\\
&=& \frac{1}{4}\int_{\mathbb{R}^3}(|\nabla u_{n}|^2+V(\varepsilon_n x+\varepsilon_n y_{\varepsilon_n})u_{n}^2)\,dx+
\Big(\frac{q-4}{4q}\Big)\int_{\mathbb{R}^3}|u_{n}|^q\,dx \\
& := &  \alpha_n
\end{eqnarray*}
hence,
\begin{eqnarray*}\limsup_{n\rightarrow\infty}\alpha_n= \limsup_{n\rightarrow\infty}c_{\varepsilon_n}\leq C(y^*).
\end{eqnarray*}

On the other hand, using Fatou's Lemma,
\begin{eqnarray*}\liminf_{n\rightarrow\infty}\alpha_n &\geq &
\frac{1}{4}\int_{\mathbb{R}^3}(|\nabla \tilde{u}|^2+V(y^*)\tilde{u}^2)\,dx+ \Big(\frac{q-4}{4q}\Big)\int_{\mathbb{R}^3}|\tilde{u}|^q\,dx\\
&= & I_{y^*}(\tilde{u})-\frac{1}{4}I'_{y^*}(\tilde{u})\tilde{u}\\
&\geq & C(y^*)
\end{eqnarray*}
then, $\displaystyle \lim_{n\rightarrow\infty}\alpha_n=C(y^*)$.

Therefore, since $\tilde{u}$ is the weak limit of $(u_n)$ in $H^1(\mathbb{R}^3)$, we conclude that $u_n\rightarrow \tilde{u}$ strongly in
$H^1(\mathbb{R}^3)$.

In particular, we have
\begin{eqnarray}\label{*}
\lim_{R\rightarrow\infty}\int_{|x|\geq R}u_n^{2^*}\,dx=0 \quad \mbox{uniformly on}\,\, n.
\end{eqnarray}

Applying Proposition \ref{R1} with $b(x)=V(\varepsilon_n x+\varepsilon_n y_{\varepsilon_n})+\varepsilon_n^2 K(\varepsilon_n x+\varepsilon_n
 y_{\varepsilon_n})\phi_{u_n}$, we obtain $u_n\in L^t(\mathbb{R}^3)$, $t\geq 2$ and
\begin{eqnarray*}\|u_n\|_t\leq C\|u_n\|
\end{eqnarray*}
where $C$ does not depend on $n$.

Now consider
\begin{eqnarray*}
-\Delta u_n &\leq & -\Delta u_n+V(\varepsilon_n x+\varepsilon_n y_{\varepsilon_n})u_n+\varepsilon_n^2 K(\varepsilon_n x+\varepsilon_n y_{\varepsilon_n})
\phi_{u_n}u_n\\
&=&|u_n|^{q-2}u_n := g_n(x).
\end{eqnarray*}

For some $t>3$, $\|g_n\|_{\frac{t}{2}}\leq C$, for all $n$. Using Proposition \ref{R2}, we have
\begin{eqnarray*}\sup_{B_{R}(y)}u_n \leq C \Big(\|u_n\|_{L^2(B_{2R}(y))}+\|g_n\|_{L^{\frac{t}{2}}(B_{2R}(y)) }\Big)
\end{eqnarray*}
for every $y\in\mathbb{R}^3$, which implies that $\|u_n\|_{L^\infty(\mathbb{R}^3)}$ is uniformly bounded. Then, from (\ref{*}),
\begin{eqnarray*}\lim_{|x|\rightarrow\infty}u_n(x)=0 \quad \mbox{uniformly on}\,\, n\in\mathbb{N}.
\end{eqnarray*}

Consequently, there exists $\varepsilon^*>0$ such that
\begin{eqnarray*}\lim_{|x|\rightarrow\infty}u_\varepsilon(x)=0 \quad \mbox{uniformly on}\,\, \varepsilon\in(0, \varepsilon^*).
\end{eqnarray*}

\end{proof}

To finish the proof of Theorem \ref{principal}, it remains to show that the solutions of $(\mathcal{SP}_\varepsilon)$ have at most one local (hence global)
maximum point $y^*$ such that $C(y^*)=\min_{\xi\in\mathbb{R}^3} C(\xi)$.

From the previous Lemma, we can focus our attention only in a fixed ball $B_R(0)\subset \mathbb{R}^3$. If $w\in L^{\infty}(\mathbb{R}^3)$ is the limit in
 $C^2_{loc}(\mathbb{R}^3)$ of
\begin{eqnarray*}w_n(x)=u_n(x+y_n)
\end{eqnarray*}
then, from Gidas, Ni and Nirenberg \cite{Gidas-Ni-Nirenberg},  $w$ is radially symmetric and has a unique local maximum at zero which is a non-degenerate global maximum.
Therefore, there exists $n_0\in\mathbb{N}$ such that $w_n$ does not have two critical points in $B_R(0)$ for all $n\geq n_0$. Consider
$p_\varepsilon\in\mathbb{R}^3$ this local (hence global) maximum of $w_\varepsilon$.

Recall that if $u_\varepsilon$ is a solution of $(S_\varepsilon)$, then
\begin{eqnarray*}v_\varepsilon(x)=u_\varepsilon(\frac{x}{\varepsilon})
\end{eqnarray*}
is a solution of $(\mathcal{SP}_\varepsilon)$.

Since $p_\varepsilon$ is the unique maximum of $w_\varepsilon$ then, $\hat{y_\varepsilon}=p_\varepsilon+y_\varepsilon$ is the unique maximum of
$u_\varepsilon$. Hence, $\tilde{y}_\varepsilon=\varepsilon p_\varepsilon+\varepsilon y_\varepsilon$ is the unique maximum of $v_\varepsilon$.

Once $p_\varepsilon\in B_R(0)$, that is, it is bounded, and $\varepsilon y_\varepsilon\rightarrow y^*$, we have
\begin{eqnarray*}\tilde{y}_\varepsilon\rightarrow y^*.
\end{eqnarray*}
where $C(y^*)=\inf_{\xi\in\mathbb{R}^3}C(\xi)$. Consequently, the concentration of functions $v_\varepsilon$ approache to $y^*$.

\paragraph{\textbf{Acknowledgements}} This paper was carried out while the author was visiting the Mathematics Department of USP in São Carlos. The author would like to thank the members of ICMC-USP for their hospitality, specially Professor S. H. M. Soares for enlightening discussions and helpful comments.
\bibliography{<your-bib-database>}

\begin{thebibliography}{00}

\bibitem{Alves-Soares} C. O. Alves, S. H. M. Soares,
\textit{Existence and concentration of positive solutions for a class of gradient systems},
Nonlinear Differ. Equ. Appl., \textbf{12}, 437-457, 2005.

\bibitem{Alves-Souto-Soares} C.O. Alves, S.H.M. Soares, M.A.S. Souto,
\textit{Schr\"{o}dinguer-Poisson equations without Ambrosetti-Rabinowitz condition},
J. Math. Anal. Appl., \textbf{377} (2011), 584-592.

\bibitem{Ambrosetti} A. Ambrosetti,
\textit{On Schrödinger-Poisson systems},
Milan J. Math., \textbf{76} (2008), 257-274.

\bibitem{Azzollini-Pomponio-SM} A. Azzollini, A. Pomponio,
\textit{Ground state solutions for the nonlinear Schrödinger-Maxwell equations},
J. Math. Anal. Appl., \textbf{345} (2008), 90-108.

\bibitem{Benci-Fortunato-1998} V. Benci, D. Fortunato,
\textit{An eigenvalue problem for the Schrödinger-Maxwell equations},
Top. Meth. Nonlinear Anal, \textbf{11}, (1998) 283-293.

\bibitem{Brezis-Kato} H. Brezis, T. Kato ,
\textit{Remarks on the Schrödinger operator with singular complex potentials},
J. Math. pures et appl., \textbf{58}, (1979) 137-151.

\bibitem{Cerami-Vaira} G. Cerami, G. Vaira,
\textit{Positive solutions for some non-autonomous Schrödinger-Poisson systems},
J. Differential Equations, \textbf{248} (2010), 521-543.

\bibitem{Coclite} G. M. Coclite,
\textit{A multiplicity result for the nonlinear Schrödinger-Maxwell equations},
Commun. Appl. Anal., \textbf{7} (2003), 417-423.

\bibitem{D'Aprile-Mugnai} T. D'Aprile, D. Mugnai,
\textit{Solitary waves for nonlinear Klein-Gordon-Maxwell and Schrodinger-Maxwell equations},
Proc. R. Soc. Edinb., Sect. A \textbf{134} (2004), 1-14.

\bibitem{Fang-Zhang} Y. Fang, J. Zhang,
\textit{Multiplicity of solutions for the nonlinear Schrödinger-Maxwell system},
Commun. Pure App. Anal., \textbf{10} (2011), 1267-1279.

\bibitem{Gidas-Ni-Nirenberg} B. Gidas, Wei-Ming Ni, L. Nirenberg,
\textit{Symmetry and related Properties via the Maximum Principle},
Commun. Math. Phys., \textbf{68} (1979), 209-243.

\bibitem{Gilbarg-Trudinger} D. Gilbarg, N. S. Trudinger,
\textit{Elliptic Partial Differential Equations of Second Order, Second edition},
Grundlehrem der Mathematischen Wissenschaften [Fundamental Principles of Mathematical Sciences], (224). Springer-Verlag, Berlin, 1983.

\bibitem{He-Zou2012} X. He, W. Zou,
\textit{Existence and concentration of ground states for Schrödinger-Poisson equations with critical growth},
J. Math. Phys., \textbf{53} (2012), 023702.

\bibitem{Ianni-Vaira} I. Ianni, G. Vaira,
\textit{On concentration of positive bound states for the Schrödinger-Poisson problem with potentials},
Adv. Nonlinear Studies, \textbf{8}, 573-595, 2008.

\bibitem{Kikuchi} H. Kikuchi,
\textit{On the existence of a solution for a elliptic system related to the Maxwell-Schrödinger equations},
Nonlinear Anal., \textbf{67} (2007), 1445-1456.

\bibitem{Lions2} P. Lions,
\textit{The concentration-compactness principle in the calculus of variations. The locally compact case. II},
Ann. Inst. H. Poincaré Anal. Non Linéaire, \textbf{1} (1984), 223-283.

\bibitem{Mercuri} C. Mercuri,
\textit{Positive solutions of nonlinear Schrödinger-Poisson systems with radial potentials vanishing at infinity},
Atti Accad. Naz. Lincei Cl. Sci. Fis. Mat. Natur. Rend. Lincei Mat. Appl., \textbf{19} (2008), 211-227.

\bibitem{Rabinowitz} P.H. Rabinowitz,
\textit{On a class of nonlinear Schrödinger equations},
Z. Angew. Math. Phys., \textbf{43}, 270-291, 1992.

\bibitem{Ruiz} D. Ruiz,
\textit{The Schrödinger-Poisson equation under the effect of a nonlinear local term},
J. Funct. Anal., \textbf{237} (2006), 665-674.

\bibitem{Wang} X. Wang,
\textit{On concentration of positive solutions bounded states of nonlinear Schrödinger equations},
Comm. Math. Phys., \textbf{153}, 229-244, 1993.

\bibitem{Wang-Zeng} X. Wang, B. Zeng,
\textit{On concentration of positive bound states of nonlinear Schrödinger equations with competing potential functions},
SIAM J. Math. Anal., \textbf{28}, 633-655, 1997.

\bibitem{Willem} M. Willem,
"Minimax theorems",
Progress in Nonlinear Differential Equations and their Applications, 24, Birkh\"auser, Boston, 1996.

\bibitem{Yang-Han} M-H Yang, Z-Q Han,
\textit{Existence and multiplicity results for the nonlinear Schrödinger-Poisson systems},
Nonlinear Anal., \textbf{13} (2012), 1093-1101.

\bibitem{Zhao-Liu-Zhao} L. Zhao, H. Liu, F. Zhao,
\textit{Existence and concentration of solutions for the Schrödinger-Poisson equations with steep potential},
J. Differential Equations, \textbf{255} (2013), 1-23.

\end{thebibliography}

\end{document}